\documentclass{amsart}
\usepackage{amsmath}%
\usepackage{amsfonts}%
\usepackage{amssymb}%
\usepackage{mathtools}%
\usepackage{graphicx}
\usepackage{color}
\newtheorem{theorem}{Theorem}
\theoremstyle{plain}

\newtheorem{definition}{Definition}

\newtheorem{lemma}{Lemma}

\newtheorem{proposition}{Proposition}
\newtheorem{remark}{Remark}

\numberwithin{equation}{section}
\begin{document}
\title[Cauchy Problem for High-Order parabolic equations]{Homotopy Regularization for a High-Order parabolic equation}

\author{P. \'Alvarez-Caudevilla}
\address[P. \'Alvarez-Caudevilla]{Universidad Carlos III de Madrid, Av. Universidad 30, 28911 Legan\'es (Madrid), Spain}%
\email[P. \'Alvarez-Caudevilla]{pcaudev@math.uc3m.es}%
\author{A. Ortega}
\address[A. Ortega]{Universidad Carlos III de Madrid, Av. Universidad 30, 28911 Legan\'es (Madrid), Spain}
\email[A. Ortega]{alortega@math.uc3m.es}%
\thanks{This paper has been partially supported by the Ministry of Economy and Competitiveness of
Spain under research project MTM2016-80618-P}
\date{\today}
\subjclass[2010]{Primary 35K30, 35K55, 35K65; Secondary 31B30} %
\keywords{High-Order parabolic equation, Homotopy deformation, Regularization arguments.}%

\begin{abstract}
In this work we study the solvability of the Cauchy Problem for a quasilinear degenerate high-order parabolic equation
\begin{equation*}
        \left\{
        \begin{tabular}{lcl}
        $u_t=(-1)^{m-1}\nabla\cdot(f^n(|u|)\nabla\Delta^{m-1}u)$ & &in $\mathbb{R}^N\times\mathbb{R}_+$, \\
        $u(x,0)=u_0(x)$& & in $\mathbb{R}^N$,
        \end{tabular}
        \right.
\end{equation*}
with $m\in\mathbb{N},\ m>1$ and $n>0$ a fixed exponent. Moreover, $f$ 
is a continuous monotone increasing positive bounded function with $f(0)=0$ and the initial data $u_0(x)$ is bounded smooth and compactly supported. 
Thus, through an homotopy argument based on an analytic $\varepsilon$-regularization of the degenerate term $f^n(|u|)$
we are able to extract information about the solutions inherited from the polyharmonic equation when $n=0$. 
\end{abstract}
\maketitle

%%%%%%%%%%%%%%%%%%%%%%%%%%%%%%%%%%%%%%%%%%%%%%%%%%%%%%%%%%%%%
\section{Introduction and main result.}
In this work we study the solvability of the Cauchy Problem for a quasilinear degenerate high-order parabolic equation of the form 
\begin{equation}\label{HGPP}
        \left\{
        \begin{tabular}{lcl}
        $u_t=(-1)^{m-1}\nabla\cdot(f^n(|u|)\nabla\Delta^{m-1}u)$ & &in $\mathbb{R}^N\times\mathbb{R}_+$, \\
        $u(x,0)=u_0(x)$& & in $\mathbb{R}^N$,
        \end{tabular}
        \right.
\end{equation}
with $m\in\mathbb{N},\ m>1$ and $n>0$ is a fixed exponent, $f$ is a continuous monotone increasing positive bounded function with $f(0)=0$ and the initial data $u_0(x)$ is a bounded smooth compactly supported function.

The principal issue to overcome in this paper is to detect proper solutions to the Cauchy Problem for the degenerate equation \eqref{HGPP} by uniformly parabolic analytic $\varepsilon$-regularizations. 

To that end, following the work \cite{Cau1}, we use an analytic homotopy approach based on \textit{a priori} estimates for solutions to uniformly parabolic analytic $\varepsilon$-regularization equations, namely
\begin{equation}\label{RHGPP}
        \left\{
        \begin{tabular}{lcl}
        $u_t=(-1)^{m-1}\nabla\cdot(\phi_{\varepsilon}(u)\nabla\Delta^{m-1}u)$ & &in $\mathbb{R}^N\times\mathbb{R}_+$, \\
        $u(x,0)=u_0(x)$& & in $\mathbb{R}^N$,
        \end{tabular}
        \right.
\end{equation}
where $\phi_{\varepsilon}(u)$, $\varepsilon\in(0,1]$ is an analytic $\varepsilon$-regularization such that $\phi_0(u)=f^n(|u|)$ and $\phi_1(u)=1$ using a classic technique relying on integral identities for weak solutions.

Next, we study an analytic homotopy transformation in both parameters, $\varepsilon\to0^+$ and $n\to0^+$ and describe branching of solutions to \eqref{HGPP} from the polyharmonic heat equation
\begin{equation}
\label{PHE33}
        \left\{
        \begin{tabular}{lcl}
        $u_t=-(-\Delta)^m u$ & &in $\mathbb{R}^N\times\mathbb{R}_+$, \\
        $u(x,0)=u_0(x)$& & in $\mathbb{R}^N$,
        \end{tabular}
        \right.
\end{equation}
which provides some qualitative oscillatory properties as well the uniqueness of solutions to \eqref{HGPP}, at least for small $n>0$. The case $m=2$ and $f(t)=t$ has been studied in \cite{Cau1}, 
however, in this paper we generalize the degenerate term $f^n(|u|)$ and under some assumptions we are able to perform an homotopy argument which provides us with the unique solutions to \eqref{HGPP} at least when the parameter $n$ is very close to zero.

In particular, we perform the homotopic deformation assuming that the fixed parameter $n>0$ is small enough. To this end, we say that \eqref{HGPP} is homotopic to the linear polyharmonic heat equation 
\eqref{PHE33}
if there exists a family of uniformly parabolic equations (the homotopic deformation) with a coefficient,
\begin{equation*}
\phi_{\varepsilon}(u)>0,\quad\mbox{analytic in both variables}\ u\in\mathbb{R}, \varepsilon\in(0,1],
\end{equation*}
with unique analytic solutions $u_{\varepsilon}(x,t)$ of the problem
\begin{equation*}
        \left\{
        \begin{tabular}{lcl}
        $u_t=(-1)^{m-1}\nabla\cdot(\phi_{\varepsilon}(u)\nabla\Delta^{m-1}u)$ & &in $\mathbb{R}^N\times\mathbb{R}_+$, \\
        $u(x,0)=u_0(x)$& & in $\mathbb{R}^N$,
        \end{tabular}
        \right.
\end{equation*}
such that $\phi_1(u)=1$ and $\phi_{\varepsilon}(u)\to f^n(|u|)$ uniformly on compact sets as $\varepsilon\to 0^+$. Based on the ideas of \cite{Cau1} we choose the homotopic path to be
\begin{equation}\label{REG1}
\phi_{\varepsilon}(u)=f^n(\varepsilon)+(1-\varepsilon)f^n\left((\varepsilon^2+u^2)^{1/2}\right).
\end{equation}
Moreover, using classic parabolic theory (see for instance \cite{Eid, Fri}) the non-degenerate equation \eqref{RHGPP} has a unique classical solution $u_{\varepsilon}(x,t)$ analytic in the variables $\varepsilon,x,t$. Thus, as it is noted in \cite{Cau1}, the homotopic deformation is a continuous deformation from solutions to \eqref{HGPP} to solutions to \eqref{PHE} for which important information is inherited 
such as uniqueness, oscillatory properties (changing sign) as well as the solvability.

Therefore, we can know define what a proper solution is in the following terms. 
\begin{definition} We say that $u(x,t)$ is a \textit{proper} solution to the Cauchy Problem \eqref{HGPP} if 
\begin{equation}\label{limit}
u_{\varepsilon}(x,t)\to u(x,t),\ \mbox{as}\ \varepsilon\to 0^+,
\end{equation}
where $\{u_{\varepsilon}(x,t)\}_{\varepsilon\in(0,1]}$ is the family of classical global solutions to the regularized Cauchy Problem \eqref{RHGPP}
\end{definition}
As we will see, due to the similarity of the expressions for weak solutions to the Cauchy Problem \eqref{HGPP} and the Free Boundary Problem corresponding to the evolution of the support of the solution of \eqref{HGPP} our analysis is unable to distinguish both type of solutions. Another issue that we will be unable to solve, due to the nature of the term $f(|u|)$, is whether the limit of $u_{\varepsilon_k}(x,t)$ can be taken independent of the subsequence $\varepsilon_k\to0^+$. In the case $f(t)=t$, thanks to the scaling properties of $f(t)$, this problem is studied with an affirmative conclusion; see \cite{Cau1}. Also, we can not discard the dependence of the solution from the type of analytic $\varepsilon$-regularization $\phi_{\varepsilon}(u)$. Hence, we must carry out alternative arguments which could solve some of the issues explained above.

Subsequently, after this limit procedure in the $\varepsilon$-regularization we perform a second limit as $n\to0^+$. That is, a continuous connection with solutions to the polyharmonic heat equation 
\eqref{PHE33}. Thus, $u(x,t)$ in \eqref{limit} is a solution to \eqref{HGPP} if 
\begin{equation*}
u(x,t)\to u_{PH}(x,t),\quad\mbox{as}\ n\to0^+.
\end{equation*}
Finally, we perform a double limit $n,\varepsilon\to0^+$ from which we obtain the conditions on the parameters $\varepsilon$ and $n$ needed to obtain such a functional convergence. 
As we have said, performing that limit over integral identities defining weak solutions results inconclusive 
to determine proper solutions to the Cauchy Problem from those to the FBP. To carry out this step we choose the simpler path
\begin{equation*}
\phi_{\varepsilon}(u)=f^n\left((\varepsilon^2+u^2)^{1/2}\right).
\end{equation*}
Now we state the main result of the paper which will provide us the definition of a proper solution of the Cauchy Problem \eqref{HGPP}.
\begin{theorem}\label{th}
Suppose that 
\begin{equation}
n|\ln f(\varepsilon(n))|\to0,\quad\mbox{as}\ n\to0^+,
\end{equation}
and the regularization family $\{u_{\varepsilon}(x,t)\}_{\varepsilon\in(0,1]}$ is uniformly bounded. Then 
\begin{enumerate}
\item The solution $u(x,t)$ to the regularized problem 
\begin{equation*}
        \left\{
        \begin{tabular}{lcl}
        $u_t=(-1)^{m-1}\nabla\cdot(f^n\left((\varepsilon^2+u^2)^{1/2}\right)\nabla\Delta^{m-1}u)$ & &in $\mathbb{R}^N\times\mathbb{R}_+$, \\
        $u(x,0)=u_0(x)$& & in $\mathbb{R}^N$,
        \end{tabular}
        \right.
\end{equation*}
converges uniformly to the solution $u_{PH}(x,t)$ to the polyharmonic heat equation \eqref{PHE} as $n\to0^+$ and $\varepsilon\to0^+$.
\item If the convolution 
\begin{equation*}
\varphi(x,t)=-\int_0^t\nabla H(x,t-s)\ast\ln|u_{PH}(x,s)|\nabla\Delta^{m-1}u_{PH}(x,s)ds,
\end{equation*}
remains bounded for the solution to the polyharmonic heat equation \eqref{PHE}, the rate of convergence as $n\to0^+$ of the asymptotic expansion\break $u(x,t)=u_{PH}(x,t)+V$ is given by
\begin{equation*}
V:=n\varphi+o(n).
\end{equation*}
\end{enumerate}
\end{theorem}
Thanks to the previous theorem we can assert that there exists a branch of solutions to the high-order equation \eqref{HGPP} emanating at $n=0^+$ from the unique solution of the polyharmonic heat equation \eqref{PHE}.
\subsection{Comparison between Cauchy Problem and Free Boundary Problem}
For both problems, the Cauchy problem and the FBP corresponding to the evolution of the support of the solution of \eqref{HGPP}, we assume that the solutions satisfy the following standard free boundary conditions:
\begin{equation}\label{boundary}
        \left\{
        \begin{tabular}{ll}
        $u=0$ & zero-height, \\
        $\nabla u=\Delta u=\nabla\Delta u=\ldots=\Delta^{m-1}=0$& zero contact angle,\\
				$\overline{\textbf{n}}\cdot\left(f^n(|u|)\nabla\Delta^{m-1}u\right)=0$& zero-flux 
        \end{tabular}
        \right.
\end{equation}
at the interface $\Gamma_0[u]$, i.e., the lateral boundary
\begin{equation*}
supp\ u\subset\mathbb{R}^N\times\mathbb{R}_+.
\end{equation*}
Due to the zero-flux condition the total mass,
\begin{equation*}
M(x,t):=\int u(x,t)dx
\end{equation*}
is preserved, since differentiating under the integral sign with respect to the temporal variable and using the Divergence Theorem,

\begin{equation*}
\frac{d}{dt}M(x,t)=(-1)^{m-1}\int_{\Gamma_0\cap\{t\}}\overline{\textbf{n}}\cdot\left(f^n(|u|)\nabla\Delta^{m-1}u\right)d\sigma=0.
\end{equation*}

%%%%%%%%%%%%%%%%%%%%%%%%%%%%%%%%%%%%%%%%%%%%%%%%%%%%%%%%%%%%%%%%%%%%%%%%%%%%%%%%%%%%%%%%%%%
\section{polyharmonic heat equation when $n=0$}
To study the solvability of the Cauchy Problem \eqref{HGPP} we use an analytic homotopic deformation from \eqref{HGPP} to an equation that provides us some useful information of its solutions, namely, to the polyharmonic heat equation,
\begin{equation}\label{PHE}
        \left\{
        \begin{tabular}{lcl}
        $u_t=-(-\Delta)^m u$ & &in $\mathbb{R}^N\times\mathbb{R}_+$, \\
        $u(x,0)=u_0(x)$& & in $\mathbb{R}^N$.
        \end{tabular}
        \right.
\end{equation}
This equation has been extensively studied in the last years \cite{Bar, Gaz2, Glk1}. It is well know that for smooth compactly supported initial data $u_0(x)$, satisfying a growth condition at infinity, see \cite{Eid},
\begin{equation}\label{weightedLp}
u_0(x)\in L_{\rho}^{2}(\mathbb{R}^N),\quad \rho(x)=e^{a|x|^{\alpha}},
\end{equation}
for some constant $a>0$ and $\alpha=\frac{2m}{2m-1}$, the polyharmonic heat equation \eqref{PHE} admits an unique classic solution given by the Poisson-type integral,
\begin{equation*}
u_{PH}(x,t)=\mathcal{H}(x,t)\ast u_0(x)=t^{-\frac{N}{2m}}\int_{\mathbb{R}^N}F\left((x-z)t^{-\frac{1}{2m}}\right)u_0(z)dz,
\end{equation*}
where $\mathcal{H}(x,t)$ is the fundamental solution for \eqref{PHE},
\begin{equation*}
\mathcal{H}(x,t)=t^{-\frac{N}{2m}}F\left(\frac{x}{t^{\frac{1}{2m}}}\right),
\end{equation*}
such that the rescaled kernel $F(y)$, with $y=\frac{x}{t^{\frac{1}{2m}}}$, is the unique radial solution of the elliptic equation
\begin{equation}\label{operator}
\mathcal{L}[F]\equiv-(-\Delta)^{m}F+\frac{1}{2m}y\cdot\nabla F+\frac{N}{2m}F=0,\quad \mbox{in}\ \mathbb{R}^N,\quad \int_{\mathbb{R}^N}F(y)dy=1. 
\end{equation}
It can be seen, \cite{Eid}, that the profile function $F(y)$ decays exponentially at infinity. Specifically, there exists some positive constants $C>1$, $a>0$ depending on $N$ and $m$ such that 
\begin{equation*}
|F(y)|\leq C\omega e^{-a|y|^{\alpha}},\quad\mbox{in}\ \mathbb{R}^N,\ \alpha=\frac{2m}{2m-1}\ \mbox{and}\ \omega=\int_{\mathbb{R}^N}e^{-a|y|^{\alpha}}dy.
\end{equation*}
On the other hand, using the Fourier Transform (see for instance \cite{Car,Eid}) the profile $F(y)$ is also given by the expression 
\begin{equation}\label{FS}
F(y)=F_{m,N}(y)=|y|^{1-N}\int_{0}^{\infty}e^{-s^{2m}}(|y|s)^{\frac{N}{2}}J_{\frac{N-2}{2}}(|y|s)ds,
\end{equation}
where $J_k$ is the $k$-th Bessel function of first kind. Note that thanks to \eqref{FS} and contrary to what happens in the case $m=1$ where the profile function is the well known Gaussian function, we know that the kernel $F_{m,N}(y)$ depends not only on the parameter $m$ but also on the dimension $N$. 

Moreover, due to the presence of the Bessel functions in the integral expression of $F(y)$, the solutions to the polyharmonic heat equation are oscillatory functions.
Another big difference between the case $m=1$ and $m>1$.  While in the first case the positivity of the solutions is preserved, this is no longer true for solutions to \eqref{PHE} with $m>1$. Nevertheless, those solutions exhibit what is called (see for instance \cite{Fer, Gaz}) \textit{eventual positivity}, i.e. there exists a time $T=T(u_0(x),K)>0$ such that for any compact set $K\subset \mathbb{R}^N$ and any compactly supported initial data $u_0(x)$,  
\begin{equation*}
u_{PH}(x,t)>0,\ \forall x\in K,\ \forall t>T.
\end{equation*}
Let us mention that the general case with $m > 1$, $m\in\mathbb{N}$, was stated as an open problem by Barbatis-Gazzola \cite{Bar} and recently
solved by Ferreira-Ferreira \cite{FF}. In fact, in \cite{FF} it was showed the eventual positivity for
every real number $m>1$, commonly known as \textit{eventual local positivity}.
To finish this brief exposition for some of the properties of the polyharmonic equation \eqref{PHE}, 
let us recall some facts about the spectrum of the operator $\mathcal{L}$ denoted by \eqref{operator}. 
As it is easily verified, for $m>1$ the operator $\mathcal{L}$ is not symmetric and does not admit a self-adjoint extension. Ascribing to the operator $\mathcal{L}$ the domain $H_{\rho}^{2m}(\mathbb{R}^N)$ it can be proved, see \cite{Ego,Glk2}, the following.
\begin{lemma}\label{espectro}\hfill
\begin{itemize}
\item The operator $\mathcal{L}:H_{\rho}^{2m}(\mathbb{R}^N)\mapsto L_{\rho}^{2}(\mathbb{R}^N)$ is a bounded operator with only the real point spectrum
\begin{equation*}
\sigma(\mathcal{L})=\left\{\lambda_{\beta}=-\frac{|\beta|}{2m},|\beta|=0,1,2,\ldots\right\}.
\end{equation*}
Eigenvalues $\lambda_{\beta}$ have finite multiplicity with eigenfunctions 
\begin{equation*}
\psi_{\beta}(y)=\frac{(-1)^{|\beta|}}{\sqrt{\beta!}}D^{\beta}F(y)\equiv\frac{(-1)^{|\beta|}}{\sqrt{\beta!}}\left(\frac{\partial}{\partial y_1}\right)^{\beta_1}\cdots\left(\frac{\partial}{\partial y_N}\right)^{\beta_N}F(y).
\end{equation*}
\item The set of eigenfunctions $\Phi=\{\psi_{\beta},|\beta|=0,1,2,\ldots \}$ is complete in $L_{\rho}^{2}(\mathbb{R}^N)$ 
\end{itemize}
\end{lemma}
In the classical case $m=1$, where the profile $F(y)$ is the rescaled Gaussian kernel, the eigenfunctions $\psi_{\beta}(y)$ are given by
\begin{equation*}
\psi_{\beta}(y)=e^{-\frac{|y|^2}{4}}\mathcal{H}_{\beta}(y),\quad \mathcal{H}_{\beta}(y)\equiv \mathcal{H}_{\beta_1}(y_1)\ldots \mathcal{H}_{\beta_N}(y_N),
\end{equation*}
where $\mathcal{H}_{\beta}$ denote the Hermite polynomials in $\mathbb{R}^N$. The operator $\mathcal{L}$ with the domain $H_{\rho}^{2}(\mathbb{R}^N)$, $\rho=e^{\frac{|y|^2}{4}}$, is 
self-adjoint and the eigenfunctions form an orthonormal basis in $L_{\rho}^{2}(\mathbb{R}^N)$. In \cite{Ego} it is also proved that the adjoint operator, 
\begin{equation*}
\mathcal{L}^*=-(-\Delta)^{m}-\frac{1}{2m}y\cdot\nabla,
\end{equation*}
possesses a set of eigenfunctions that forms an orthonormal basis in $L_{\rho^*}^{2}(\mathbb{R}^N)$, with the specific 
exponentially decaying weight function $\rho^*(y)=e^{-a|y|^{\alpha}}$.\newline Moreover, $\mathcal{L}^*:H_{\rho^*}^{2m}(\mathbb{R}^N)\mapsto L_{\rho^*}^{2}(\mathbb{R}^N)$ is a bounded linear operator,
\begin{equation*}
\langle \mathcal{L}[v],w\rangle=\langle v,\mathcal{L}^*[w]\rangle \quad\mbox{for any }v\in H_{\rho}^{2m}(\mathbb{R}^N),\ w\in H_{\rho^*}^{2m}(\mathbb{R}^N),
\end{equation*}
and $\sigma(\mathcal{L}^*)=\sigma(\mathcal{L})$ with the eigenfunctions $\{\psi_{\beta}^*(y)\}$ being polynomials of order $|\beta|$,
\begin{equation*}
\sqrt{\beta!}\psi_{\beta}^*(y)=y^{\beta}+\sum_{j=1}^{\left[\frac{|\beta|}{2m}\right]}\frac{1}{j!}(-\Delta)^{mj}y^{\beta}.
\end{equation*}

\medskip
%%%%%%%%%%%%%%%%%%%%%%%%%%%%%%%%%%%%%%%%%%%%%%%%%%%%%%%%%%%%%

\section{Preliminary estimates: Bernis-Friedman type-inequality.}
Throughout this section, for any $\varepsilon\in(0,1]$ let $u_{\varepsilon}(x,t)$ be the solution of Cauchy Problem for the regularized non-degenerate uniformly parabolic equation \eqref{RHGPP}. 
By classic parabolic theory \cite{Eid, Fri} this family is continuous and analytic in $\varepsilon\in(0,1]$ in the appropriate functional topology, at least in some interval $[0,T]$. Moreover, all the derivatives are H\"older continuous in $\overline{\Omega}\times[0,T]$. 
From now on, we denote with $\Omega$ either $\mathbb{R}^N$ or, equivalently, the bounded domain $supp\ u\cap\{t\}$ (the section of the support).

The following result comes from similar ideas as those performed by Bernis-Friedman \cite{Ber} and will be used in the sequel to prove some of the main results of this work.
\begin{proposition}\label{ber_estimate}
Let $u_{\varepsilon}(x,t)$ be the unique global solution to the Cauchy Problem for the regularized non-degenerate equation \eqref{RHGPP}. Then for $t\in[0,T]$, there exists $K>0$ independent of $\varepsilon$ and $T$ such that for $j\in\mathbb{N}$,
\begin{enumerate}
\item $\displaystyle\int_{\Omega}|\Delta^{\frac{m-1}{2}}u_{\varepsilon}(x,t)|^2dx\leq K$ if $m=2j+1$.
\item $\displaystyle\int_{\Omega}|\nabla\Delta^{\frac{m-2}{2}}u_{\varepsilon}(x,t)|^2dx\leq K$  if $m=2j$.
\item $\displaystyle\int_{\Omega}|\Delta^{\frac{m-2}{2}}u_{\varepsilon}(x,t)|^2dx\leq K,\ \mbox{if}\ m=2j,\ j\in\mathbb{N}$ .
\item $\displaystyle\int_{\Omega}u_{\varepsilon}(x,t)dx\leq K$.
\item Setting $h_{\varepsilon}=\phi_{\varepsilon}(u_{\varepsilon})\nabla\Delta^{m-1}u_{\varepsilon}$, we have $||h_{\varepsilon}||_{L^2(\Omega\times(0,t))}\leq K$.
\end{enumerate}
\end{proposition}
\begin{proof}
First we note that, thanks to the boundary conditions \eqref{boundary},
\begin{equation*}
-\int_{\Omega}u_{\varepsilon}(x,\cdot)\Delta^{m-1}u_{\varepsilon}(x,\cdot)dx=\left\{
        \begin{tabular}{lr}
        $\displaystyle(-1)^m\int_{\Omega}|\Delta^{\frac{m-1}{2}}u_{\varepsilon}(x,\cdot)|^2dx$ & if $m=2j+1$, \\
				&\\
				$\displaystyle(-1)^m\int_{\Omega}|\nabla\Delta^{\frac{m-2}{2}}u_{\varepsilon}(x,\cdot)|^2dx$ & if $m=2j$,
        \end{tabular}
        \right.
\end{equation*}
for $j\in\mathbb{N}$, as well as
\begin{equation*}
\int_{\Omega}u_{\varepsilon}(x,t+h)\Delta^{m-1}u_{\varepsilon}(x,t)dx=
\int_{\Omega}u_{\varepsilon}(x,t)\Delta^{m-1}u_{\varepsilon}(x,t+h)dx.
\end{equation*}
Hence,
\begin{align*}
-\int_{\Omega}&[\Delta^{m-1}u_{\varepsilon}(x,t+h)+\Delta^{m-1}u_{\varepsilon}(x,t)][u_{\varepsilon}(x,t+h)-u_{\varepsilon}(x,t)]dx=\\
&=\left\{ \begin{tabular}{lr}
        $\displaystyle(-1)^m\int_{\Omega}|\Delta^{\frac{m-1}{2}}u_{\varepsilon}(x,t+h)|^2-|\Delta^{\frac{m-1}{2}}u_{\varepsilon}(x,t)|^2dx$ & if $m=2j+1$, \\
				&\\
				$\displaystyle(-1)^m\int_{\Omega}|\nabla\Delta^{\frac{m-2}{2}}u_{\varepsilon}(x,t+h)|^2-|\nabla\Delta^{\frac{m-2}{2}}u_{\varepsilon}(x,t)|^2dx$ & if $m=2j$.
        \end{tabular}
        \right.
\end{align*}
Then, dividing by $h$, taking the limit as $h\to0^+$ and integrating between $0$ and $t\in[0,T]$ we get
\begin{align}\label{id1}
-\iint\limits_{\Omega\times(0,t)}&\Delta^{m-1}u_{\varepsilon}(x,t)u_{\varepsilon,t}(x,t)\, dx\,dt\notag\\
&=\left\{ \begin{tabular}{lr}
        $\displaystyle\frac{(-1)^m}{2}\int_{\Omega}|\Delta^{\frac{m-1}{2}}u_{\varepsilon}(x,t)|^2-|\Delta^{\frac{m-1}{2}}u_{\varepsilon}(x,0)|^2dx$ & if $m=2j+1$, \\
				&\\
				$\displaystyle\frac{(-1)^m}{2}\int_{\Omega}|\nabla\Delta^{\frac{m-2}{2}}u_{\varepsilon}(x,t)|^2-|\nabla\Delta^{\frac{m-2}{2}}u_{\varepsilon}(x,0)|^2dx$ & if $m=2j$.
        \end{tabular}
        \right.
\end{align}
Now, multiplying the regularized equation \eqref{RHGPP} by $\Delta^{m-1}u_{\varepsilon}$, integrating by parts in $\Omega\times(0,t)$ and using the boundary conditions, we obtain
\begin{equation}\label{id2}
\iint\limits_{\Omega\times(0,t)}\mkern-5mu \nabla\cdot(\phi_{\varepsilon}(u_{\varepsilon})\nabla\Delta^{m-1}u_{\varepsilon})\Delta^{m-1}u_{\varepsilon}\, dx\,dt
=\iint\limits_{\Omega\times(0,t)}\mkern-5mu\phi_{\varepsilon}(u_{\varepsilon})|\nabla\Delta^{m-1}u_{\varepsilon}|^2\, dx\,dt.
\end{equation}
Therefore, from \eqref{id1} and \eqref{id2}, we conclude 
\begin{equation*}
\int_{\Omega}|\Delta^{\frac{m-1}{2}}u_{\varepsilon}(x,0)|^2dx=\int_{\Omega}|\Delta^{\frac{m-1}{2}}u_{\varepsilon}(x,t)|^2 \, dx\,dt+2\iint\limits_{\Omega\times(0,t)}\mkern-5mu\phi_{\varepsilon}(u_{\varepsilon})|\nabla\Delta^{m-1}u_{\varepsilon}|^2\, dx\,dt
\end{equation*}
if $m=2j+1$, and 
\begin{equation*}
\int_{\Omega}|\nabla\Delta^{\frac{m-2}{2}}u_{\varepsilon}(x,0)|^2dx\!=\!\int_{\Omega}|\nabla\Delta^{\frac{m-2}{2}}u_{\varepsilon}(x,t)|^2 \, dx\,dt
+2\iint\limits_{\Omega\times(0,t)}\mkern-5mu\phi_{\varepsilon}(u_{\varepsilon})|\nabla\Delta^{m-1}u_{\varepsilon}|^2\, dx\,dt,
\end{equation*}
if $m=2j$. Consequently, due to these Bernis-Friedman-type inequalities we have proved assertions $(1)$ and $(2)$. Let us observe that from the above integral equalities we also get
\begin{equation*}
\iint\limits_{\Omega\times(0,t)}\phi_{\varepsilon}(u_{\varepsilon})|\nabla\Delta^{m-1}u_{\varepsilon}|^2\,dx\,dt\leq K,
\end{equation*}
and, therefore, 
\begin{equation}\label{cota_m}
f^n(\varepsilon)\!\iint\limits_{\Omega\times(0,t)}|\nabla\Delta^{m-1}u_{\varepsilon}|^2\,dx\,dt\leq K,
\end{equation}
\begin{equation}\label{cota_m2}
\iint\limits_{\Omega\times(0,t)}\!\!f^n\left((\varepsilon^2+u_{\varepsilon}^2)^{1/2}\right)|\nabla\Delta^{m-1}u_{\varepsilon}|^2\,dx\,dt\leq K,
\end{equation}

with $K$ as a positive constant.
For unbounded domains such as $\Omega=\mathbb{R}^N$ the results are true thanks to the exponential decay of solutions (so that the integration by parts is justified). 
Thus, the inequalities remain true in certain $L_{\rho}^2(\mathbb{R}^N)$ and $H_{\rho}^{2m}(\mathbb{R}^N)$ weighted spaces for an appropriate weight. Moreover, from the conservation of mass and the boundary conditions \eqref{boundary}, it also follows that
\begin{equation*}
\int_{\Omega}u_{\varepsilon}(x,t)dx\leq K,\ \forall t\in[0,T].
\end{equation*}
On the other hand, applying Poincar\'e's 
inequality in the case $m=2j$ (assuming a bounded domain $\Omega$) we find
\begin{equation*}
\int_{\Omega}|\Delta^{\frac{m-2}{2}}u(x,t)|^2\,dx\,dt\leq K,
\end{equation*}
and we conclude $(3)$. Finally, we prove $(5)$. Since $f$ is a bounded function, i.e. $\displaystyle\sup\limits_{t\in\mathbb{R}_+}f(t)\leq C_f$, it follows that 
\begin{align*}
\iint\limits_{\Omega\times(0,t)}|h_{\varepsilon}|^2dxdt&=\iint\limits_{\Omega\times(0,t)}\phi_{\varepsilon}^2(u_{\varepsilon})|\nabla\Delta^{m-1}u_{\varepsilon}|^2\,dx\,dt\\
&=\iint\limits_{\Omega\times(0,t)}\left(f^n(\varepsilon)+(1-\varepsilon)f^n\left((\varepsilon^2+u_{\varepsilon}^2)^{1/2}\right)\right)^2|\nabla\Delta^{m-1}u_{\varepsilon}|^2\,dx\,dt\\
&\leq 2f^{2n}(\varepsilon)\iint\limits_{\Omega\times(0,t)}\!\!\!|\nabla\Delta^{m-1}u_{\varepsilon}|^2\,dx\,dt\\
&+2C_f^n\iint\limits_{\Omega\times(0,t)}f^n\left((\varepsilon^2+u_{\varepsilon})^{1/2}\right)|\nabla\Delta^{m-1}u_{\varepsilon}|^2\,dx\,dt\\
&\leq2KC_f^n.
\end{align*}
\end{proof}
Additionally, we obtain uniform $L^{\infty}$ estimates for solutions to \eqref{HGPP} by means of a scaling technique, \cite{Gid}.
\begin{proposition}
Any solution to problem \eqref{HGPP} is uniformly bounded.
\end{proposition}
\begin{proof}
We argue by contradiction. Assume that there exists a monotone sequence $\{t_k\}\to T$ and $\{x_k\}\subset\mathbb{R}^N$ such that 
\begin{equation}
\label{ass}
\sup\limits_{(x,t)\in\mathbb{R}^N\times(0,t_k)}|u(x,t)|=|u(x_k,t_k)|=C_k\to+\infty\quad\mbox{monotonically}.
\end{equation}
Subsequently, we rescale the solution $u(x,t)$ to \eqref{HGPP} and define the sequence $\{v_k(y,s)\}$ as follows,
\begin{equation*}
v_k(y,s):=\frac{1}{C_k}u\left(\lambda_ky+x_k,\lambda_k^{2m}s+t_k\right),
\end{equation*}
for some positive number $\lambda_k$ (to be specified later) such that $\{\lambda_k\}\to0$. Thus, with this rescaling we just perform a zoom around the point 
$(x_k,t_k)$ in the region $B_{\delta/\lambda_k}(0)\times \left(\frac{-t_k}{\lambda_k^{2m}},0\right)$, for $\delta>0$ sufficiently small and where $B_{\delta/\lambda_k}(0)$ is the ball of radius $\frac{\delta}{\lambda_k}$ and centered at the origin.
Therefore, due to the scaling and assumption \eqref{ass} it is now clear that, 
\begin{equation}\label{acotado}
|v_k(0,0)|=1\quad\mbox{and}\quad |v_k(y,s)|\leq1,\quad\mbox{for all }k\geq1\mbox{ and }s\in\left[-\frac{t_k}{\lambda_k^{2m}},0\right).
\end{equation}
Moreover, the function $v_k$ satisfies the equation
\begin{equation}\label{eq:rescale}
\frac{\partial}{\partial s}v_k=(-1)^{m-1}\nabla\cdot(f^n(|C_kv_k|)\nabla\Delta^{m-1}v_k),
\end{equation}
for any $(y,s)\in\mathbb{R}^N\times(-\frac{t_k}{\lambda_k^{2m}},0)$ with initial data $\displaystyle v_{k0}(y)=\frac{1}{C_k}u_0\left(\lambda_ky+x_k\right)$. On the other hand, thanks to the uniform estimate $(1)$ in Proposition \ref{ber_estimate}, for a positive constant $K$, we obtain
\begin{equation*}
\int_{\Omega}|\Delta^{\frac{m-1}{2}}u(x,t)|^2dx=\frac{C_k^2}{\lambda_k^{N+2(m-1)}}\int_{\Omega_k}|\Delta^{\frac{m-1}{2}}v_{k}(y,s)|^2dy\leq K,
\end{equation*}
so that
\begin{equation*}
\int_{\Omega_k}|\Delta^{\frac{m-1}{2}}v_{k}(y,s)|^2dy\leq \frac{\lambda_k^{N+2(m-1)}}{C_k^2}K,
\end{equation*}
if $m=2j+1$. In a similar way, if $m=2j$, from $(2)$ in Proposition \ref{ber_estimate} we find,
\begin{equation*}
\int_{\Omega_k}|\nabla\Delta^{\frac{m-2}{2}}v_{k}(y,s)|^2dy\leq \frac{\lambda_k^{N+2(m-1)}}{C_k^2}K.
\end{equation*}
Moreover, using $(3)$ in Proposition \ref{ber_estimate}, 
\begin{equation*}
\int_{\Omega_k}|\Delta^{\frac{m-2}{2}}v_{k}(y,s)|^2dy\leq \frac{\lambda_k^{N+2(m-1)}}{C_k^2}K_1.
\end{equation*}

Hence, passing to the limit as $k\to\infty$, along a subsequence if necessary, the limit function $v_k\to v(y,s)$ satisfies, 
\begin{equation}\label{cota}
\int_{\mathbb{R}^N}|\Delta^{\frac{m-1}{2}}v(y,s)|^2dy=0\quad\mbox{if }m=2j+1,
\end{equation}
and
\begin{equation}\label{cota1}
\int_{\mathbb{R}^N}|\Delta^{\frac{m-2}{2}}v(y,s)|^2dy=0\quad\mbox{if }m=2j.
\end{equation}
Therefore, passing to the limit and using \eqref{cota} and \eqref{cota1} together with the boundary conditions \eqref{boundary}, we find that the limit function satisfies
\begin{equation*}
\left\{
\begin{tabular}{rll}
$\Delta^{\widetilde{m}}v$&$\!\!\!\!=0$ & in $\mathbb{R}^N$, \\
$|v|$&$\!\!\!\!\leq1$,&\\
$\lim\limits_{|y|\to\infty}v(y,\cdot)$&$\!\!\!\!=0$. &
\end{tabular}
\right.
\end{equation*}
with $\widetilde{m}=\frac{m-1}{2}$ if $m=2j+1$ and $\widetilde{m}=\frac{m-2}{2}$.
Therefore, because of a Liouville-type Theorem, see \cite{Arm,Lig}, we obtain that $v$ has to be constant, and due to the condition at infinity we conclude that $v\equiv0$ in contradiction with \eqref{acotado}.
 Consequently, we conclude, from the construction of the functions $v_k$ and the limiting argument 
performed above that $v\equiv0$ in contradiction with \eqref{acotado}. Actually, \eqref{acotado} implies, by interior parabolic regularity, that $v(y,0)$ must be 
non-trivial in a neighbourhood of $y=0$. Then, this fact simply means that problem \eqref{HGPP} does not have an internal mechanism to support infinite growth (or blow-up) solutions. 
\end{proof}
\section{Homotopy deformations}
Next we show the existence of solutions to the Cauchy Problem \eqref{HGPP} using a limiting argument as
\begin{itemize}
\item $\varepsilon\to0^+$, obtaining the convergence of solutions to the regularized problem \eqref{RHGPP} to solutions to problem \eqref{HGPP}.
\item $\varepsilon\to0^+$ and $n=n(\varepsilon)\to0^+$, obtaining the convergence of solutions to problem \eqref{HGPP} to solutions to the polyharmonic heat equation \eqref{PHE} under some conditions on the behavior of $n(\varepsilon)$ for $\varepsilon\approx 0$.
\end{itemize}
As the former procedure is unable to distinguish proper solutions to the Cauchy problem \eqref{HGPP} from solutions to the FBP we perform a second homotopic argument as 
\begin{itemize}
\item  $n\to0^+$ and $\varepsilon=\varepsilon(n)\to0^+$ as $n\to0^+$.
\end{itemize}
First we recall the following Lemma due to Aubin and Lions, see \cite{Aub}.
\begin{lemma}\label{AL}
Let $X_0\subseteq X\subset X_1$ be three Banach spaces such that $X_0$ is compactly embedded in $X$ and $X$ is continuously embedded in $X_1$. For $1\leq p,q\leq\infty$, let 
$$W=\{u\in L^p([0,T],X_0),\,u_t\in L^q([0,T],X_1)\}.$$
\begin{itemize}
\item If $p<\infty$ then the embedding of $W$ into $L^p([0,T],X)$ is compact.
\item If $p=\infty$ and $q>1$ then the embedding of $W$ into $C([0,T],X)$ is compact.
\end{itemize}
\end{lemma}
First, for bounded domains $\Omega$ and due to 
Proposition \ref{ber_estimate} together with Lemma \ref{AL} we can extract a convergent subsequence in $L^2(\Omega\times[0,T])$ as $\varepsilon\to0^+$ so that 
\begin{equation*}
u_{\varepsilon}(x,t)\to u(x,t),\quad \mbox{in}\ L^2(\Omega\times[0,T]),\quad\mbox{as}\ \varepsilon\to0^+,
\end{equation*}
with $u(x,t)$ a solution of \eqref{HGPP}. Thereby, the convergence is strong in $L^2(\Omega\times[0,T])$. In the whole space $\mathbb{R}^N$ we use the appropriate $L_{\rho}^2(\mathbb{R}^N)$ and $H_{\rho}^{2m}(\mathbb{R}^N)$ weighted spaces. Note that the difficult issue, that we do not overcome at this stage, is whether the limit depends on the taken subsequence, in other words, if the limit as $\varepsilon\to0^+$ provides a unique limit or many partial limits.
\begin{lemma}Let $u_{\varepsilon}(x,t)$ be the unique global solution of the regularized problem \eqref{RHGPP}, then
\begin{equation*}
||u_{\varepsilon}(\cdot,t)-u(x,\cdot) ||_{L^2(\Omega\times(0,t)}\rightarrow0,\ \mbox{as}\ \varepsilon\to0^+,
\end{equation*}
with $u(x,t)$ a solution of \eqref{HGPP}, i.e.,
\begin{equation*}
\iint_{\Omega\times(0,t)}\varphi_tudxdt+(-1)^m\iint_{\Omega\times(0,t)}\nabla\varphi\left(f^n(u)\nabla\Delta^{m-1}u\right)\,dx\,dt=0,
\end{equation*}
for all $\varphi\in C_0^{\infty}(\Omega\times(0,t)),t\in[0,T].$
\end{lemma}
\begin{proof}
Multiplying equation \eqref{RHGPP} by a test function $\varphi\in C_0^{\infty}(\Omega\times(0,t))$ and integrating by parts we get
\begin{equation*}
\iint_{\Omega\times(0,t)}\varphi_tu_{\varepsilon}dxdt+(-1)^m\iint_{\Omega\times(0,t)}\nabla\varphi\left(\phi_{\varepsilon}(u)\nabla\Delta^{m-1}u_{\varepsilon}\right)\,dx\,dt=0.
\end{equation*}
Substituting $\phi_{\varepsilon}(u)=f^n(\varepsilon)+(1-\varepsilon)f^n\left((\varepsilon^2+u_{\varepsilon}^2)^{1/2}\right)$ into the latter equation we find,
\begin{align}\label{est}
&\iint_{\Omega\times(0,t)}\varphi_tu_{\varepsilon}dxdt+(-1)^mf^n(\varepsilon)\iint_{\Omega\times(0,t)}\nabla\varphi\cdot\nabla\Delta^{m-1}u_{\varepsilon}\,dx\,dt\\
&+(-1)^m(1-\varepsilon)\iint_{\Omega\times(0,t)}f^n\left((\varepsilon^2+u_{\varepsilon}^2)^{1/2}\right)\nabla\varphi\cdot\nabla\Delta^{m-1}u_{\varepsilon}\,dx\,dt=0.\nonumber
\end{align}
Now, we focus on controlling the second term in \eqref{est}. To do so, we use Proposition\,\ref{ber_estimate} together with the H\"older's inequality, and we find,
\begin{align*}
&\left|f^n(\varepsilon)\iint_{\Omega\times(0,t)}\nabla\varphi\cdot\nabla\Delta^{m-1}u_{\varepsilon}\,dx\,dt\right|\\
&\leq f^n(\varepsilon)\left(\iint_{\Omega\times(0,t)}|\nabla\Delta^{m-1}u_{\varepsilon}|^2dxdt\right)^{\frac{1}{2}}\left(\iint_{\Omega\times(0,t)}|\nabla\varphi|^2\,dx\,dt\right)^{\frac{1}{2}}\\
&\leq C f^{\frac{n}{2}}(\varepsilon)\left(f^n(\varepsilon)\iint_{\Omega\times(0,t)}|\nabla\Delta^{m-1}u_{\varepsilon}|^2\,dx\,dt\right)^{\frac{1}{2}}\\
&\leq Kf^{\frac{n}{2}}(\varepsilon)\to0,\quad \mbox{as}\ \varepsilon\to0^+
\end{align*}
with $K$ a positive constant. 
To control the third term we split the integration domain in the following sets
\begin{equation*}
\mathcal{G}_{\varepsilon,\delta}:=\{(x,t)\in\Omega\times(0,t): |u_{\varepsilon}(x,t)|>\delta>0\},
\end{equation*}
and
\begin{equation*}
\mathcal{B}_{\varepsilon,\delta}:=\{(x,t)\in\Omega\times(0,t): |u_{\varepsilon}(x,t)|\leq\delta\},
\end{equation*}
for any fixed arbitrarily small $\delta>0$. In the uniform non-degeneracy set $\mathcal{G}_{\varepsilon,\delta}$ it is clear that the limiting solution as $\varepsilon\to0^+$ is a weak solution of \eqref{HGPP}. Also, by parabolic regularity for the uniformly parabolic equation \eqref{RHGPP}, we get that $u_{\varepsilon,t}$ and $\phi_{\varepsilon}(u_{\varepsilon})\nabla\Delta^{m-1}u_{\varepsilon}$ converge in compact subsets of $\mathcal{G}=\mathcal{G}_{0,0}$. Thus, as it happens in \cite{Ber} and \cite{Cau1} we obtain that the limit function $u(x,t)=\lim_{\varepsilon\to0^+}u_{\varepsilon}(x,t)$ satisfies
\begin{equation}\label{goodset}
\iint_{\mathcal{G}}\varphi_t u\;dxdt+(-1)^m\iint_{\mathcal{G}}\nabla\varphi\left(f^n(|u|)\nabla\Delta^{m-1}u\right)dxdt=0.
\end{equation}
Then, the limit function $u(x,t)$ is a solution to the Cauchy Problem \eqref{HGPP}. Nevertheless, in the set of parabolic degeneracy $\mathcal{B}_{\varepsilon,\delta}$, we have to take $\varepsilon>0$ small enough and depending on $\delta$. Indeed, let $0<\varepsilon\leq\delta$ fixed. Applying the H\"older's inequality to the third term in \eqref{est} in the set $\mathcal{B}_{\varepsilon,\delta}$ and using that $f$ 
is a continuous monotone increasing function, we find
\begin{align}\label{conv}
&\left|\iint_{\mathcal{B}_{\varepsilon,\delta}}\nabla\varphi f^n\left((\varepsilon^2+u_{\varepsilon}^2)^{1/2}\right)\nabla\Delta^{m-1}u_{\varepsilon}dxdt\right|\\
&\leq C\left(\iint_{\mathcal{B}_{\varepsilon,\delta}} f^{2n}\left((\varepsilon^2+u_{\varepsilon}^2)^{1/2}\right)|\nabla\Delta^{m-1}u_{\varepsilon}|^2dxdt\right)^{\frac{1}{2}}\nonumber\\
&\leq Cf^{\frac{n}{2}}\left((\varepsilon^2+\delta^2)^{1/2}\right)\left(\iint_{\mathcal{B}_{\varepsilon,\delta}} f^{n}\left((\varepsilon^2+u_{\varepsilon}^2)^{1/2}\right)|\nabla\Delta^{m-1}u_{\varepsilon}|^2dxdt\right)^{\frac{1}{2}}\nonumber\\
&\leq K f^{\frac{n}{2}}\left((\varepsilon^2+\delta^2)^{1/2}\right)\to0\nonumber,
\end{align}
provided $\delta\to0$ as $\varepsilon\to0^+$. Therefore, the integration over the set of degeneracy has no distinguishable effects respect to the integration over the sets $\mathcal{G}_{\varepsilon,\delta}$ in the final limit. 
Thus, the limit as $\varepsilon\to0$ provides weak solutions to \eqref{HGPP}. 
\end{proof}
\begin{remark}
Due to the boundary conditions \eqref{boundary} the weak formulation \eqref{goodset} 
also holds for solutions to the FBP, so that our analysis is unable to distinguish solutions to the Cauchy problem from those to the FBP.
\end{remark}
Now we perform the limit when $n\to0^+$. Let us notice that the estimate provided by \eqref{conv} reflects the rate of convergence if we perform a second homotopic limit as $n\to0$, together with $\varepsilon\to0$, in the analytic regularization \eqref{REG1}, in order to obtain weak solutions emanating from the polyharmonic heat equation \eqref{PHE}.\newline 
To get such a functional convergence we need $n=n(\varepsilon)\to0^+$ such that, for $\delta\approx\varepsilon$,  
\begin{equation*}
f^{n(\varepsilon)}\left(\varepsilon\right)\to0,\quad\mbox{as } \varepsilon\to0^+,
\end{equation*}
that is,
\begin{equation}\label{est11}
n(\varepsilon)\ln f(\varepsilon)\to-\infty,\quad\mbox{as } \varepsilon\to0.
\end{equation}
Hence, we need $n=n(\varepsilon)$ such that 
\begin{equation}\label{est12}
n(\varepsilon)>>\frac{1}{|\ln f(\varepsilon)|},
\end{equation}
that will provide us with the convergence, at least in a weak sense, of solutions. Thus, under this hypotheses we arrive at a solution of the polyharmonic heat equation \eqref{PHE} 
as $\varepsilon,n(\varepsilon)\to0$, written in the very-weak form,
\begin{equation*}
\iint_{\Omega\times(0,t)}\varphi_t u\;dxdt+(-1)^{m}\iint_{\Omega\times(0,t)}\nabla\varphi\cdot\nabla\Delta^{m-1}u\;dxdt=0.
\end{equation*}
Let us remark that this is not a full definition of weak solution since it just assumes a single integration by parts, so that performing the limit as $\varepsilon,n(\varepsilon)\to0$ allows us to obtain, among other things a solution of \eqref{PHE} under the boundary conditions \eqref{boundary} (with n=0). 
It is clear now that applying this limiting argument in the integral identities does not allow us to ascertain any difference between CP-solutions and FBP-solutions.

Consequently, a stronger version of our homotopic arguments is indispensable to identify correctly the proper solutions to the Cauchy problem \eqref{HGPP}. 

Nonetheless, this homotopic approach provides us with estimates and bounds such as \eqref{est11} and \eqref{est12} which are necessary for a correct limiting process.
 Moreover, keeping in mind the oscillatory nature of the kernel $F(|y|)$ of the polyharmonic heat equation, inevitably, the proper solutions to \eqref{HGPP} are going to be oscillatory near the interface provided $n>0$ is small enough.
\subsection{Branching of solutions}\hfill\break
Next we analyze the double limit as $n\to0^+$ and $\varepsilon\to0^+$. As a consequence we obtain the solvability of the equation \eqref{HGPP} through a homotopy deformation from solutions to the polyharmonic heat equation \eqref{PHE} (which are oscillatory) to solutions to problem \eqref{HGPP}.

To do so, we now consider the regularization 
$$\psi_{\varepsilon}(u)=f^n\left((\varepsilon^2+u)^{1/2}\right),$$ 
and therefore we will handle the following regularized equation
\begin{equation}\label{eqnp}
u_t=(-1)^{m-1}\nabla\cdot\left(\psi_{\varepsilon}(u)\nabla\Delta^{m-1}u\right),
\end{equation}
with smooth compactly supported initial data. Due to parabolic estimates we may assume that $u_{\varepsilon}(x,t)$ decays exponentially at infinity. Moreover, now  we take
\begin{equation*}
n\to0^+,
\end{equation*}
as the principal deformation parameter and then we will choose the appropriate 
\begin{equation*}
\varepsilon=\varepsilon(n)\to0^+.
\end{equation*}
Next, we rewrite equation \eqref{eqnp} as 
\begin{equation*}
u_t=-(-\Delta)^mu+(-1)^{m-1}\nabla\cdot\left([1-\psi_{\varepsilon}(u)]\nabla\Delta^{m-1}u\right),
\end{equation*}
that in terms of the fundamental solution for \eqref{PHE} can be written as
\begin{equation}\label{int1}
u(x,t)=\mathcal{H}(x,t)\ast u_0(x)+\int_{0}^{t}\nabla \mathcal{H}(x,t-s)\ast \Theta_{n,\varepsilon}(u(x,s))\nabla\Delta^{m-1}u(x,s)ds,
\end{equation}
where $\Theta_{n,\varepsilon}(u)=1-\psi_{\varepsilon}(u)$. The convergence to the well posed polyharmonic heat equation \eqref{PHE} 
will strongly depend on the weak limit of the second term of \eqref{int1}, i.e., on the behaviour of
\begin{equation}\label{precondition}
\Theta_{n,\varepsilon}(u)=1-\psi_{\varepsilon}(u)=1-f^n\left((\varepsilon^2+u^2)^{1/2}\right)\to0,\quad\mbox{as}\ n,\varepsilon(n)\to0^+.
\end{equation}
Thus, to carry out such a branching analysis we need to verify the following expansion:
\begin{equation}\label{condition}
\Theta_{n,\varepsilon}(u)=-n\ln f\left((\varepsilon^2+u^2)^{1/2}\right)(1+o(n)),\quad\mbox{as}\ n,\to0^+.
\end{equation}
 on a fixed family of uniformly bounded smooth solutions $\{u_{\varepsilon}(x,t)\}$. Note that, 
 checking \eqref{condition} in the sets $\mathcal{B}_{\varepsilon,\delta}$, i.e., where $u\approx0$, requires the condition
\begin{equation}\label{hyp}
n\left|\ln f\left(\varepsilon(n)\right)\right|\to0,\quad\mbox{as }n\to0^+.
\end{equation}
This will be the principal assumption on the parameter $\varepsilon(n)$ and its relation with $n$, in order to guarantee such convergence of solutions.
\begin{proof}[Proof of Theorem \ref{th}.]
Under the condition \eqref{hyp} we perform a branching analysis following the steps performed in \cite{Cau1}. Substituting \eqref{condition} in \eqref{int1}, we find,
\begin{align}\label{int2}
u(x,t)=&\mathcal{H}(x,t)\ast u_0(x)\\
       -&n\int_{0}^{t}\nabla \mathcal{H}(x,t-s)\ast \ln f\left(\right(\varepsilon^2+u^2)^{1/2})\nabla\Delta^{m-1}u(x,s)ds+o(n^2).\nonumber
\end{align}
Now, we take 
$$u=u_{PH}(x,t)+n\varphi+o(n),$$
with $u_{PH}(x,t)$ a solution to the polyharmonic heat equation \eqref{PHE} and $\varphi$ to be determined. Thus, substituting into \eqref{int2}, and omitting terms of high order we obtain
\begin{align*}
&u_{PH}(x,t)+n\varphi=\mathcal{H}(x,t)\ast u_0(x)\\
&\!-\! n\!\!\!\int_{0}^{t}\!\!\!\nabla \mathcal{H}(x,t-s)\!\ast \ln\! f\!\left(\right(\varepsilon^2\!\!+ u_{PH}^2(x,s)\!\!+2nu_{PH}\varphi\!+\! n^2\varphi)^{1/2})\nabla\Delta^{m-1}(u_{PH}(x,s))ds.
\end{align*}
Passing to the limit as $n\to0^+$ we get the following expression for the error function
\begin{equation}\label{int3}
\varphi=\int_{0}^{t}\nabla \mathcal{H}(x,t-s)\ast \ln f\left(|u_{PH}|\right)\nabla\Delta^{m-1}u_{PH}ds.
\end{equation}
The asymptotic expansion assumes that \eqref{int3} is always finite, i.e.
\begin{equation*}
\ln f\left(|u_{PH}|\right)\in L_{loc}^1(\mathbb{R}^N),
\end{equation*}
for any $t>0$, so $f\left(|u_{PH}|\right)$ does not have zeros with an exponential decay in some neighbourhood. 
In particular, this is true if the solutions have transversal zeros. Observe that to obtain \eqref{condition} from \eqref{precondition}, we have to use the expansion for small $n>0$,
\begin{equation}\label{exp1}
1-|f|^n\equiv 1-e^{n\ln |f|}=1-(1+n\ln|f|+\ldots)=n\ln|f|+\ldots,
\end{equation}
which is true pointwise on any set $\{f\ge c_0\}$ for an arbitrarily small fixed constant $c_0>0$. However, in a small neighborhood 
of any zero of $f\left(|u_{PH}|\right)$, the expansion \eqref{exp1} is no longer true. Nevertheless, it remains true in a weak sense provided that this  zero is sufficiently transversal in a natural sense, i.e.,
\begin{equation*}
  \frac{1-|f|^n}n \rightharpoonup -\ln|f|,\quad\mbox{as } n\to 0^+
\end{equation*}
in $L^\infty_{\rm loc}$. Although this fact is rather plausible, as it is noted in \cite{Cau1}, there is not a rigorous proof 
for general solutions to the polyharmonic heat equation. Therefore, we include such assumptions in our argument.

Finally, we have to check that the perturbation $\Theta_{n,\varepsilon}(u)$ is small, which is guaranteed by the following.
\begin{itemize}
\item[(1)] At one hand, thanks to the uniform estimate \eqref{cota_m2}, using the Young inequality for convolutions, we find that $\Theta_{n,\varepsilon}(u)\to0$ as $n,\varepsilon(n)\to0^+$ for the domain $\{|u|\geq t_{1}\}$ with 
\begin{equation*}
|\ln t|\leq cf^{\frac{n}{2}}(t),\quad\mbox{with } t\geq t_{1},
\end{equation*}
for some constant $c>0$.
\item[(1b)] Observe that, in a similar way as above, thanks to the uniform estimate for $h_{\varepsilon}$ in Proposition \ref{ber_estimate}, we find that $\Theta_{n,\varepsilon}(u)\to0$ as $n,\varepsilon(n)\to0^+$ for the domain $\{|u|\geq t_{2}\}$ with 
\begin{equation*}
|\ln t|\leq cf^n(t),\quad\mbox{with } t\geq t_{2}.
\end{equation*}
\item[(2)] On the other hand, consider the integral equality \eqref{int1} in the domain where
\begin{equation*}
\mathcal{D}_{i,\varepsilon}\equiv\{\varepsilon^2\leq\varepsilon^2+u^2\leq t_{i}\},\quad i=1,2.
\end{equation*}
The maximal singularity of the term $\ln f\left((\varepsilon^2+u^2)^{1/2}\right)$ in the domain $\mathcal{D}_{i,\varepsilon}$ is achieved when $u=0$. 
Therefore, it is of order $O(\ln f(\varepsilon))$ and, hence, the perturbation term has order at most $O\left(n\ln f(\varepsilon)\right)$. Then, because of \eqref{hyp} we conclude
\begin{equation*}
O\left(n\ln f(\varepsilon)\right)\to0,\quad\mbox{as}\ n\to0^+.
\end{equation*}
\end{itemize}

\end{proof}

Let us stress that the representation $u=u_{PH}(x,t)+n\varphi+o(n)$ provided by Theorem \ref{th}, requires the convergence 
of \eqref{int3} as $n\to0^+$ which is difficult to verify for arbitrary solutions to the polyharmonic heat equation \eqref{PHE}. Nevertheless, thanks to the regularity of solutions such an integral divergence due to the formation of flat zeros can occur at a finite number of points, so it is expected at least almost everywhere.
As it happens in the case $m=2$ and $f(t)=t$, see \cite{Cau1}, solutions to \eqref{HGPP} are those which can be deformed as $n\to0^+$ through the analytic path $\psi_{\varepsilon}(u)$ to the unique solution to the polyharmonic heat equation with same initial data. Therefore, according to our development, a suitable setting of the Cauchy problem for the high order problem \eqref{HGPP} requires the whole set of solutions $\{u(x,t):\ n>0\}$ or the two-parameter set $\{u_{\varepsilon}(x,t):\ n>0,\varepsilon>0\}$ of regularized solutions. Hence, this approach results useless to treat an individual problem of type \eqref{HGPP} for a fixed $n>0$. Nonetheless, it provides qualitative properties for solutions to problem \eqref{HGPP} inherited from those solutions to the polyharmonic heat equation \eqref{PHE}.

Finally, we observe that, due to the nature of the nonlinear term $f$ we are unable to provide a conclusive answer to whether
\begin{equation*}
\lim\sup u_{\varepsilon}(x,t)=\lim\inf u_{\varepsilon}(x,t),\quad\mbox{as }\varepsilon\to0^+.
\end{equation*} 
In the case $m=2$ and $f(t)=t$, studied in \cite{Cau1}, the proof of such equality relies on the homogeneity properties of the non linear term $f(t)=t$. In fact, the proof follows studying an auxiliary problem independent of $\varepsilon$ obtained by means of a scaling in the space variables for the regularized problem \eqref{RHGPP}. Therefore these arguments automatically extends to the case of consider $f(t)=t^{\kappa}$ for $\kappa>0$. Hence, as the one-variable homogeneous functions are such a power functions, this ideas does not work when one considers a general nonlinearity $f$.

\end{document}